\documentclass[11pt,reqno]{amsart}
\usepackage{latexsym,amsmath,amssymb,mathrsfs,xcolor,amsthm}
\usepackage{enumerate}
\usepackage{graphicx}
\usepackage{units}
\usepackage{ esint }
\usepackage{comment}
\usepackage{url}

plus 6pt minus 12pt
\newtheorem{definition}{Definition}[section]

\newtheorem{lemma}{Lemma}[section]
\newtheorem{remark}{Remark}[section]
\newtheorem{conjecture}{Conjecture}[section]
\newtheorem{corollary}{Corollary}[section]

\numberwithin{equation}{section}

\renewcommand{\d}{\,{\rm d}}

\newtheorem{theorem}{Theorem}[section]
\title[Mean Value Estimates]{Mean Value Estimates for a Real-Exponent Analogue of Waring's Problem}
\author{Ataleshvara Bhargava}

\begin{document}

\begin{abstract} For non-integer $\theta > 3$ and $\kappa \geq 1$, we show that the smallest $r_0$ such that the mean value estimate 
\[ \int_{-\kappa}^{\kappa} \Big| \sum_{X < x \leq 2X} e(\alpha x^{\theta}) \Big|^{2r} \d\alpha \ll_{\epsilon} \kappa X^{2r - \theta+\epsilon}  \]
holds for all integers $r \geq r_0$ satisfies $2r_0 \leq \theta^2(1+O(\theta^{-1/2}))$. This is an improvement over the previous bound by Poulias \cite{Poulias1} of $2r_0 \leq (\lfloor 2\theta \rfloor + 1)(\lfloor 2\theta \rfloor + 2)$. As a consequence, the bound on the asymptotic order of the minimum number of variables required to prove the expected asymptotic formula for the number $R_{s,\theta}(N)$ of solutions $(x_1,\ldots,x_s) \in \mathbb{N}^s$ to the Diophantine equation 
\[ \lfloor x_1^{\theta} \rfloor+\cdots+\lfloor x_s^{\theta} \rfloor = N \]
is improved by a factor of $4$. We also discuss a certain Diophantine system which arises naturally from our proof, which may have applications to other counting problems and may be of independent interest. 
\end{abstract}

\maketitle 

\everymath{\displaystyle} 

\section{Introduction}
First considered by Deshouillers \cite{Des} in 1973 and Arkhipov and Zhitkov \cite{ArkhZhit} in 1984, there has recently been renewed interest in analogues of Waring's problem with non-integral exponent; see for example \cite{Bhargava2, Kuf, mad26, Poulias1, Poulthesis, taf26}. An important tool in such problems is a mean value estimate of the form 
\begin{align} \label{eq:meanval}
\int_{-\kappa}^{\kappa} \Big| \sum_{X < x \leq 2X} e(\alpha x^{\theta}) \Big|^{2r} \d\alpha \ll \kappa X^{2r - \theta+\epsilon} 
\end{align}
for large $X > 0$, for a fixed non-integer $\theta > 1$, where $\kappa \geq 1$, and $r$ is an integer sufficiently large in terms of $\theta$. Estimates of this type were first implicitly proven by Deshouillers \cite{Des} under the assumption $\theta > 12$ and $2r > c_1\theta^3(\log(\theta)+14)$ and improved by Arkhipov and Zhitkov \cite{ArkhZhit} to $\theta > 12$ and $2r > c_2\theta^2\log(\theta) $ for positive constants $c_1$ and $c_2$. Using the estimates available with the now proven Main Conjecture in Vinogradov's Mean Value Theorem (see \cite{BDG} or \cite{WooNEC}, see also \cite{WooleyCube} for the first proof of the cubic case), Poulias \cite{Poulias1} improved these bounds to $\theta > 2$ and $2r \geq (\lfloor 2\theta \rfloor + 1)(\lfloor 2\theta \rfloor + 2)$. Interestingly, the analogous estimates when $\theta$ is replaced by an integer $k$ hold when $2r \geq k^2+k$, and so the bounds for this non-integer variant differ from what may be true by at least a factor of $4$. One of our main objectives here is to reduce this discrepancy for large $\theta$ so that, up to lower order terms, the number of variables required for both is the same. 

\begin{theorem} \label{thm:main}
Fix a non-integer $\theta > 3$. Then \eqref{eq:meanval} holds for $2r \geq \theta^2+9\theta^{3/2}$. 
\end{theorem}

Our proof actually gives a slightly sharper bound on the lower order terms for large $\theta$. In fact, for $\theta > 3$ one only needs 
\[2r \geq \theta^2 +\frac{7}{3}\theta^{3/2}+6\theta+\frac{20}{3}\theta^{1/2}, \] 
and for $\theta > 4$ one only needs 
\[ 2r \geq \theta^2+\frac{2^{11/4}}{3} \theta^{3/2}+(3+\frac{5\sqrt{2}}{2})\theta+2^{1/4}(\frac{11}{3}+\frac{3\sqrt{2}}{2})\theta^{1/2}. \]

Our method is an iterative process of compressing the ranges of variables. We discuss how the secondary terms might be further reduced using this process. The stipulation $\kappa \geq 1$ in \eqref{eq:meanval} is just a matter of convenience; the same estimate holds as long as $\kappa \geq c$ for some constant $c$ independent of $X$. We also note that by Theorem \ref{thm:main} and a simple interpolation argument, \eqref{eq:meanval} holds for all $p \in \mathbb{R}$ satisfying $2\lfloor p \rfloor \geq \theta^2+9\theta^{3/2}$ as well. However, our main focus is on even integer exponents since these are the exponents that appear in the applications we consider.  We will derive Theorem \ref{thm:main} as a special case of a more general result. We introduce the notation 
\[ g(\alpha) = g(\alpha;X) = \sum_{X < x \leq 2X} e(\alpha x^{\theta}) \]
to be used for the rest of this paper. We begin with a definition. 

\begin{definition} \label{def:diopair}
Fix $m, k \in \mathbb{N}$ with $k \geq 2$. We call $(m,k)$ an admissible Diophantine pair (which we shorten to just admissible or Diophantine pair when the context is clear) if for any positive constants $C_0, \ldots, C_{k-1}$ the set $\mathcal{C}_{m,k}(Y)$ of integer solutions $(x_1,\ldots,x_{2m})$ to the system 
\begin{align} \label{eq:diopairsystem}
|x_1^i +\cdots+x_{m}^i-x_{m+1}^i-\cdots-x_{2m}^i| \leq C_i Y^{i(1-1/k)} \quad \text{ for } \quad 1 \leq i \leq k-1 
\end{align}
with $|x_i| \leq C_0Y$ 
satisfies $|\mathcal{C}_{m,k}(Y)| \ll_{C_0,\ldots,C_{k-1}, \epsilon} Y^{2m-\frac{1}{2}(k-1)+\epsilon}$ for any $\epsilon > 0$. 
\end{definition} 

The analysis of Diophantine pairs plays a key role in our proof. We will discuss Diophantine pairs further at the end of the paper. Theorem \ref{thm:main} follows from the following result. 
\begin{theorem} \label{thm:general}
Suppose $\theta > 3$ and $3 \leq n \leq \lfloor \theta \rfloor$. Let $\{(m_k,k)\}_{k=3}^{n}$ be a sequence of admissible Diophantine pairs and let $M_n = m_3+\ldots+m_n$. 
Then \eqref{eq:meanval} holds for large $X > 0$ as long as \[ 2r \geq \lceil \theta(1-1/n)^{-1} \rceil(\lceil \theta(1-1/n)^{-1}\rceil+1)+2M_n. \] 
\end{theorem}

The most immediate applications of these mean value estimates are to Diophantine counting problems of large non-integer degree. For real $\Lambda > 0$, non-integer $\theta > 1$, fixed $\tau > 0$ and integer $s$, define $R_{s,\theta,\tau}(\Lambda)$ to be the number of solutions $(x_1,\ldots,x_s) \in \mathbb{N}^s$ to the Diophantine inequality 
\begin{align} \label{eq:maindioph}
|x_1^{\theta}+\cdots+x_s^{\theta}-\Lambda| < \tau. 
\end{align}
Let $\widetilde{G}_0(\theta)$ be the smallest number of variables which for the expected asymptotic formula for $R_{s,\theta,\tau}(\Lambda)$ holds as $\Lambda \to \infty$. 
Poulias \cite{Poulias1} proves that $\widetilde{G}_0(\theta) \leq (\lfloor 2\theta \rfloor+1)(\lfloor 2\theta \rfloor +2)+1$.  
Theorem \ref{thm:main} improves the main term in the bound on $\widetilde{G}_0(\theta)$ by a factor of $4$. 
\begin{corollary} \label{cor:Gbound}
For $\theta > 3$, one has $\widetilde{G}_0(\theta) \leq \theta^2+9\theta^{3/2}+1$. 
\end{corollary}
This bound follows by using Theorem \ref{thm:main} in place of Theorem 3.4 in \cite{Poulias1} and proceeding exactly in the same way as done in that paper. 
We note that our methods also lead to further improvements in the lower order terms in the above bound for $\widetilde{G}_0(\theta)$. However, in the interest of reducing complexity our focus in Corollary \ref{cor:Gbound} has been in optimizing for the main term. We discuss possible further improvements at the end of the paper. 

Similarly, write $R_{s,\theta}(N)$ to be the number of solutions to the equation 
\[ \lfloor x_1^{\theta}\rfloor+\cdots+\lfloor x_s^{\theta} \rfloor = N, \]
and let $\widetilde{G}_1(\theta)$ be the minimum number of variables required to prove the expected asymptotic formula for $R_{s,\theta}(N)$ as $N \to \infty$. 
This equation was first considered by Deshouillers \cite{Des}, who proved the existence of solutions as long as 
\[ s \geq 4\theta \log\theta+2\log\log\theta+14, \]
and that $\widetilde{G}_1(\theta) \leq 6\theta^{3}(\log(\theta)+14)$. Arkhipov and Zhitkov \cite{ArkhZhit} reduced this latter bound to to $\widetilde{G}_1(\theta) \leq 22\theta^2(\log(\theta)+4)$. Using the methods of \cite{Poulias1} and applying the now proven Main Cojecture of Vinogradov's Mean Value Theorem, one may further reduce the required number of variables to $\widetilde{G}_1(\theta) \leq (\lfloor 2\theta \rfloor+1)(\lfloor 2\theta \rfloor +2)+1$. Using Theorem \ref{thm:main} in place of Lemma 3 in \cite{ArkhZhit} and proceeding as done in that paper, one has the following. 
\begin{corollary} \label{cor:Gbound2}
For $\theta > 3$, one has $\widetilde{G}_1(\theta) \leq \theta^2+9\theta^{3/2}+1$. 
\end{corollary}

Diophantine counting problems of non-integer degree have a decades long history and have seen a revival of interest in recent years. There has recently been additional interest in extending results from additive number theory, such as those which hold for classical Waring bases, to the real exponent setting. Madritsch \cite{mad26} has proven asymptotic formulas for representations by a more general class of pseudo-polynomials. The author proves the existence of almost-diagonal solutions to \eqref{eq:maindioph} in \cite{Bhargava2}. K\"{u}fner \cite{Kuf} considers counting problems with varying exponents and proves an effective version of the Fre\u iman-Scourfield Theorem for real-exponents, an analog to the works of Fre\u iman \cite{Freiman}, Scourfield \cite{Scour}, and Br\"{u}dern and Wooley \cite{BWFrei}. T\'{a}fula \cite{taf26} has proven the existence of economical asymptotic subbases of the sequence $\{ \lfloor x^{\theta} \rfloor: x \in \mathbb{N} \}$ for $\theta > 1$ and $s \geq 5$ for $\theta \in (1,2)$ and $s \geq (\lfloor 2\theta \rfloor+1)(\lfloor 2\theta \rfloor +2)+1$ for $\theta > 2$. The results and methods of this paper might be useful in further reducing the number of variables required to prove several of these results. Indeed, using Theorem \ref{thm:main} in the argument of \cite{Bhargava2} immediately allows one to show the existence of almost-diagonal solutions to \eqref{eq:maindioph} for $s \geq \theta^2 +9\theta^{3/2}+1$. The approach to this problem is similar to that of Corollary \ref{cor:Gbound}. 

The paper is organized as follows. In Section \ref{sec:prelim} we prove some auxiliary lemmas which will be used in Section \ref{sec:meanval} to prove Theorems \ref{thm:main} and \ref{thm:general}. In Section \ref{sec:conc} we discuss possible improvements to the methods and results of this paper, and also a related conjecture on Diophantine pairs. As a quick word on notation, we use Vinogradov's notation $A \ll B$ and $A = O(B)$ in the usual sense, and write $A \asymp B$ to mean $A \ll B$ and $B \ll A$. We use the standard shorthand $e(\alpha) = e^{2\pi i \alpha}$. We use $\lceil t \rceil $ to denote the smallest $m \in \mathbb{Z}$ such that $t \leq m$ and $\lfloor t \rfloor$ to denote the largest integer not exceeding $t$.

\subsection{Acknowledgments} The author is grateful to Prof. T.D. Wooley, for many enlightening discussions and helpful suggestions and for funding support from NSF grant DMS-2502625.

\section{Preliminary Lemmas} \label{sec:prelim}
The following lemma is an extension of Lemma 1 from \cite{ArkhZhit}. Write 
\[ b_i = \binom{\theta}{i}. \]

\begin{lemma} \label{lem:arkzhit}
Let $\theta > 1$ and let $0 \leq \gamma \leq 1/2$. Let $m$ be an integer with $m \geq \theta(1-\gamma)^{-1} $, and let $0 < P \in \mathbb{R}$ be large enough. Let $r$ be the unique integer satisfying $\theta(1-\gamma)^{-1}-1 < r \leq \theta(1-\gamma)^{-1}$, and define $s_0 = \lfloor \gamma^{-1}-1\rfloor $. Define 
\[\mathcal{H} = b_1P^{\theta-1}x_1+\cdots+b_mP^{\theta-m}x_m. \]
Then for any positive real number $t$, the number $\mathcal{L}(P)$ of solutions $(x_1, \cdots, x_m)$ with $|x_i| \ll_{i, \theta} tP^{\gamma i}$ to the inequality 
\[ |\mathcal{H}| \leq t \]
satisfies 
\[ \mathcal{L}(P) \ll_{\theta} t^mP^{\frac{\gamma}{2}m(m+1)-\theta+ \Upsilon}, \]
where 
\[ \Upsilon = -\frac{\gamma}{2}s_0(s_0+1)+(s_0+1)(1-\gamma) \quad \text{ if } \quad s_0+1 < r \]
and 
\[ \Upsilon = -\frac{\gamma}{2}r(r+1)+r \quad \text{ if } \quad s_0+1 \geq r. \] 
\end{lemma}

\begin{proof} The proof is very similar to the proof of Lemma 1 in \cite{ArkhZhit}. If $|\mathcal{H}| \leq t$, then we must have 
\[ b_1P^{\theta-1}x_1+\cdots+b_mP^{\theta-m}x_m = \eta t \]
for some $|\eta| \leq 1$. 
As a result, one has 
\[ x_1 = b_1^{-1}(-b_2x_2P^{-1}-\cdots-b_mx_mP^{1-m}+t\eta P^{1-\theta}) = b_1^{-1} \Big( t\eta P^{1-\theta} - \sum_{i=2}^m b_ix_iP^{-i+1} \Big). \]
Since $|x_i| \ll tP^{\gamma i}$ for $i = 1, \ldots, m$, we thus have 
\[ |x_1| \ll t|b_1|^{-1} \Big( P^{1-\theta} + \sum_{i=2}^m |b_i|P^{1+(\gamma-1)i} \Big) \ll t(P^{2\gamma-1}+P^{1-\theta}). \]
Since $\theta > 1$ and $2\gamma -1 \leq 0$, the right hand side above is $\ll t$, and so for large $P$, we deduce that $x_1$ can take on at most $\ll t$ many values. 

Now, choose an integer $i > 1$. Let 
\[ A_i = -b_i^{-1}\sum_{j = 1}^{i-1} b_jP^{\theta-j}x_j. \]
Then by the fact that $\mathcal{H} = \eta t$ we must have 
\[ x_i - A_i = -b_i^{-1}\sum_{j=i+1}^m b_jP^{i-j}x_j +t\eta b_i^{-1}P^{i-\theta}, \]
so that 
\[ |x_i - A_i| \ll b_i^{-1}\Big( tP^{i-\theta} + \sum_{j=i+1}^m |b_j||x_j|P^{i-j} \Big) \ll tP^{i-\theta} + t\sum_{j=i+1}^m P^{i-j(\gamma-1)} \ll t(P^{i-\theta}+P^{\gamma(i+1)-1}). \]
Fix the choices of $x_1, \ldots, x_{i-1}$, so that $A_i$ is fixed. If $i \leq (1-\gamma)^{-1}\theta -1$, then $i- \theta \leq \gamma(i+1)-1$, so $|x_i -A_i| \ll tP^{\gamma(i+1)-1}$. If $\gamma(i+1)-1 \leq 0$, which is to say if $i \leq \gamma^{-1}-1$, there are $\ll t$ many values that $x_i$ can take once all previous variables are fixed. Otherwise, there are $\ll tP^{\gamma(i+1)-1}$ many values that $x_i$ can take. Instead, if $i > (1-\gamma)^{-1}\theta-1$, then $|x_i-A_i| \ll tP^{i-\theta} $, and therefore the number of values that $x_i$ can take is $\ll tP^{i-\theta}$. However, we also have the trivial estimate from assumption that $x_i$ can take at most $\ll tP^{\gamma i}$ many values, which is a better estimate than $\ll tP^{i-\theta}$ as long as $\gamma i \leq i-\theta$, which is to say when $\theta(1-\gamma)^{-1} \leq i$. 

Let $s_0 = \lfloor \gamma^{-1}-1\rfloor $. Since $m \geq \theta (1-\gamma)^{-1}$, we have 
\[ (1-\gamma)^{-1}\theta-1 < r \leq (1-\gamma)^{-1}\theta \leq m.\] 
We split into two cases depending on whether 
$s_0+1 < r$, or $s_0 +1 \geq r$. 
If the first inequality holds, then $s_0+1 \leq r-1$, so there is at least one integer $i \in [s_0+1,r-1]$. 
The total number of solutions is hence $\ll t^m P^{\upsilon}$, where 
\[ \upsilon = \sum_{i = s_0+1 }^{r-1} (\gamma(i+1)-1)+\sum_{i = r+1}^m i\gamma + r-\theta = \frac{\gamma}{2}( m(m+1)- s_0(s_0+1))+(s_0+1)(1-\gamma)- \theta. \]
(The sum over $r+1 \leq i \leq m$ is considered to have no summands if $m=r$). Putting together the above gives us the required estimate in the first case. 

If $s_0 +1 \geq r$ instead, then the total number of solutions is $\ll t^mP^{\upsilon}$, where 
\[ \upsilon = \sum_{i = r+1}^m i\gamma + r-\theta = \frac{\gamma}{2}( m(m+1)- r(r+1))+r- \theta. \]
\end{proof}

\begin{remark} \label{rem:1}
Suppose $\gamma = n^{-1}$ for some integer $2 \leq n \leq \lfloor \theta \rfloor$. Then 
\[ s_0 + 1 = \lfloor \gamma^{-1}-1 \rfloor+1 = n < \frac{\theta}{1-\frac{1}{n}}-1 < r, \]
so the first case of Lemma \ref{lem:arkzhit} applies, by which we conclude that 
\[ \mathcal{L}(P) \ll t^m P^{\frac{m(m+1)}{2k}+\frac{n-1}{2}-\theta}. \]
\end{remark}

The next lemma is similar to Lemma 3.2 in \cite{Poulias1} and Lemma 2.1 in \cite{Watt}. We start with some definitions. Let $1 \leq h \leq r$ and $n \in \mathbb{N} \cup \{0\}$. let $V_{r,h}^n(I_1,I_2;\delta,m)$ be the number of solutions to the system 
\begin{equation} \label{eq:append}
\begin{aligned}
& |x_1^{\theta}+\cdots+x_r^{\theta}-x_{r+1}^{\theta}-\cdots-x_{2r}^{\theta}| < \delta, 
\\ 
& \sum_{i=1}^r (x_i^n-x_{i+r}^n) = m, 
\\ 
& \sum_{i=1}^r (x_i^j-x_{i+r}^j) = 0
\quad \quad (1 \leq j \leq n-1),
\end{aligned}
\end{equation}
where $x_1, \ldots,x_h, x_{r+1}, \ldots,x_{r+h} \in I_1$ and $x_i \in I_2$ for all other $i$. We write 
\[ V_{r,h}^n(I_1,I_2;\delta,0) = V_{r,h}^n(I_1,I_2;\delta). \] If $I_1 = I_2 = I$, then we write $V_{r,h}^n(I_1,I_2;\delta,m) = V_r^n(I;\delta,m)$ and $V_r^n(I;\delta,0) = V_r^n(I;\delta)$, since $h$ is irrelevant. Next, let $ \mathcal{S} \subseteq \mathbb{Z}^{2h}$ be any finite collection of lattice points. Let $V_{r,h}^n(\mathcal{S}, I_2,;\delta,m)$ be the number of solutions to \eqref{eq:append} with 
\[(x_1,\ldots, x_h, x_{r+1},x_{r+2},\ldots,x_{r+h}) \in \mathcal{S}\] 
and $x_i \in I_2$ for all other $i$. We again write $V_{r,h}^n(\mathcal{S}, I_2,;\delta,0) = V_{r,h}^n(\mathcal{S}, I_2,;\delta)$. 
For $\alpha \in \mathbb{R}$ and $\boldsymbol{\beta}_n = (\beta_1,\ldots,\beta_n) \in \mathbb{R}^n$, define 
\[ H_i(\alpha, \boldsymbol{\beta}_n) = H(\alpha, \boldsymbol{\beta}_n; I_i) = \sum_{x \in I_i} e(\alpha x^{\theta} +\beta_1x+\cdots+\beta_nx^n ). \]
For $1 \leq j \leq n$ and $1 \leq h \leq r$, put 
\[ \sigma_{h,j}(\mathbf{x}) = \sum_{i=1}^h (x_i^j-x_{r+i}^j) \quad \text{ and } \quad \boldsymbol{\sigma}_{h,n} = (\sigma_{h,1}(\mathbf{x}), \ldots, \sigma_{h,n}(\mathbf{x})), \]
and define 
\[ H_{\mathcal{S}}(\alpha, \boldsymbol{\beta}_n) = \sum_{(x_1,\ldots,x_h,x_{r+1},\ldots,x_{r+h}) \in \mathcal{S}} \ e\Big(\alpha\sum_{i=1}^h (x_i^{\theta}-x_{r+i}^{\theta})+\boldsymbol{\beta}_n\cdot \boldsymbol{\sigma}_{h,n} \Big). \]

\begin{lemma} \label{lem:poulcount}
Define $\Delta = (2\delta)^{-1}$. Then for any $1\leq h \leq r$ and $n \in \mathbb{N}$ one has 
\[ V_{r,h}^n(\mathcal{S}, I_2,;\delta,m) \ll \delta \int_{-\Delta}^{\Delta} \int_0^1 |H_{\mathcal{S}}(\alpha,\boldsymbol{\beta}_n)H_2(\alpha,\boldsymbol{\beta}_n)^{2r-2h}| \d \boldsymbol{\beta}_n \d \alpha. \]
In particular, one has 
\[ V_{r,h}^n(I_1, I_2; \delta,m) \ll \delta \int_{-\Delta}^{\Delta} \int_0^1 |H_1(\alpha,\boldsymbol{\beta}_n)^{2h} H_2(\alpha,\boldsymbol{\beta}_n)^{2r-2h}| \d \boldsymbol{\beta}_n \d \alpha. \]
Moreover, it also holds that 
\[ V_{r,h}^n(I_1, I_2; \delta) \asymp \delta \int_{-\Delta}^{\Delta} \int_0^1 |H_1(\alpha,\boldsymbol{\beta}_n)^{2h} H_2(\alpha,\boldsymbol{\beta}_n)^{2r-2h}| \d \boldsymbol{\beta}_n \d \alpha. \]
All implicit constants are independent of $I_1$, $I_2$, $\mathcal{S}$, $\theta$ and $\delta$. 
\end{lemma}

\begin{proof} The proof is an adaptation of the proof of Lemma 3.2 from \cite{Poulias1}. We start by proving the upper bound on $V_{r,h}(\mathcal{S}, I_2; \delta) $. Define the auxiliary functions 
\[ K(\alpha) = \operatorname{sinc}^2(\alpha) \quad \text{ and } \quad \Lambda(x) = \max\{0, 1-|x|\}, \] where $\operatorname{sinc}(\alpha) = 1 $ if $\alpha =0$ and $\operatorname{sinc}(\alpha) = (\pi\alpha)^{-1}\sin(\pi\alpha)$
otherwise. We then have the identities
\[ K_0(\xi) = \int_{\mathbb{R}} e(-x\xi)\Lambda(x)\d x \quad \text{ and } \quad \Lambda(x) = \int_{\mathbb{R}} e(x\xi)K_0(\xi)\d\xi, \]
see Lemma 20.1 in \cite{DavBook}. It holds that $\pi^2K_0(\alpha)/4 \geq 1$ for $|\alpha|\leq 1/2$. For $\mathbf{x} = (x_1,\ldots,x_{2r})$, put 
\[ \sigma_{r,\theta}(\mathbf{x}) = \sum_{i=1}^r \big(x_i^{\theta}-x_{i+r}^{\theta}\big) \quad \text{ and } \quad \xi_0(\mathbf{x}) = (2\delta)^{-1} \sigma_{r,\theta}(\mathbf{x}). \]
Note that for $-\Delta \leq \alpha \leq \Delta$, one has $|\delta\alpha|\leq 1/2$. Note also that $K_0$ is even. Recall that $\boldsymbol{\beta}_n = (\beta_1,\ldots,\beta_n)$. From the estimate on $K_0$, we deduce by orthogonality that 
\begin{align*} 
V_{r,h}(\mathcal{S},I_2) & \leq \frac{\pi^2}{4} \sum_{\mathbf{x}} K_0(\xi_0(\mathbf{x})) \int_{0}^1 e(\boldsymbol{\sigma}_{r,n}\cdot\boldsymbol{\beta}_n-m\beta_n)\d\boldsymbol{\beta}_n 
\\ & = \frac{\pi^2}{4} \sum_{\mathbf{x}} \int_{-\infty}^{\infty} e(u\xi_0(\mathbf{x})) \Lambda(u)\d u \int_{0}^1 e(\boldsymbol{\sigma}_{r,n}\cdot\boldsymbol{\beta}_n-m\beta_n)\d\boldsymbol{\beta}_n 
\\
& = \frac{\pi^2\delta}{2} \sum_{\mathbf{x}} \int_{\infty}^{\infty} e(\alpha\sigma_{r,\theta}(\mathbf{x})) \Lambda(2\delta\alpha)\d \alpha \int_{0}^1 e(\boldsymbol{\sigma}_{r,n}\cdot\boldsymbol{\beta}_n-m\beta_n)\d\boldsymbol{\beta}_n, 
\end{align*}
where the sum is over all $\mathbf{x}=(x_1,\ldots,x_{2r})$ satisfying $(x_1,\ldots,x_h,x_{r+1},\ldots,x_{r+h}) \in \mathcal{S}$ and $x_i \in I_2$ for all remaining $i$. Note that 
\begin{align*} \sum_{\mathbf{x}} \int_{-\infty}^{\infty} e(\alpha\sigma_{r,\theta}(\mathbf{x})) \Lambda(2\delta\alpha)\d \alpha & \int_{0}^1 e(\boldsymbol{\sigma}_{r,n}\cdot\boldsymbol{\beta}_n-m\beta_n)\d\boldsymbol{\beta}_n 
\\ & = \int_{-\infty}^{\infty} \int_0^1 H_{\mathcal{S}}(\alpha,\boldsymbol{\beta}_n)|H_2(\alpha,\boldsymbol{\beta}_n)|^{2r-2h} e(-m\beta_n) \Lambda(2\delta\alpha) \d \boldsymbol{\beta}_n \d \alpha. 
\end{align*}
Next, $\Lambda(2\delta\alpha) = 0$ for $|\alpha| > (2\delta)^{-1} = \Delta$, and $0 \leq \Lambda(2\delta\alpha) \leq 1$ for all $\alpha$. Hence, we use the Triangle Inequality to deduce that 
\begin{align*} V_{r,h}(\mathcal{S},I_2) & \leq \frac{\pi^2\delta}{2} \int_{-\infty}^{\infty} \int_0^1 |H_{\mathcal{S}}(\alpha,\boldsymbol{\beta}_n)H_2(\alpha,\boldsymbol{\beta}_n)^{2r-2h}| \Lambda(2\delta\alpha) \d \boldsymbol{\beta}_n \d \alpha 
\\
& \leq \frac{\pi^2\delta}{2} \int_{-\Delta}^{\Delta} \int_0^1 |H_{\mathcal{S}}(\alpha,\boldsymbol{\beta}_n)H_2(\alpha,\boldsymbol{\beta}_n)^{2r-2h}| \d \boldsymbol{\beta}_n \d \alpha, 
\end{align*}
proving the first inequality. Applying this inequality with $\mathcal{S} = I_1^{2h}$ derives the second upper bound, and letting $m=0$ also derives the upper bound in the third estimate. Thus, it only remains to prove the lower bound in the third estimate 

Note that $\pi^2K_0(\delta\alpha)/4 \geq 1$ for $|\delta\alpha| \leq 1/2$, which is when $|\alpha| \leq \Delta^{-1}$. We therefore have 
\begin{align*} 
\delta \int_{-\Delta}^{\Delta} \int_0^1 |H_1(\alpha,\boldsymbol{\beta}_n)^{2h} & H_2(\alpha,\boldsymbol{\beta}_n)^{2r-2h}| \d \boldsymbol{\beta}_n \d \alpha 
\\ & \leq \frac{\pi^2\delta}{4} \int_{-\Delta}^{\Delta} \int_0^1 |H_1(\alpha,\boldsymbol{\beta}_n)^{2h}H_2(\alpha,\boldsymbol{\beta}_n)^{2r-2h}| K_0(\delta\alpha) \d \boldsymbol{\beta}_n \d \alpha 
\\ 
& \leq \frac{\pi^2\delta}{4} \int_{-\infty}^{\infty} \int_0^1 |H_1(\alpha,\boldsymbol{\beta}_n)^{2h}H_2(\alpha,\boldsymbol{\beta}_n)^{2r-2h}| K_0(\delta\alpha) \d \boldsymbol{\beta}_n \d \alpha
\\ 
& = \frac{\pi^2\delta}{4} \sum_{\mathbf{x}} \int_{-\infty}^{\infty} e(\alpha\sigma_{r,\theta}(\mathbf{x}))K_0(\delta\alpha)\d\alpha \int_0^1 e(\boldsymbol{\sigma}_{r,n}\cdot\boldsymbol{\beta}_n)\d\boldsymbol{\beta}_n
\\ 
& = \frac{\pi^2}{4} \sum_{\mathbf{x}} \Lambda(2\xi_0(\mathbf{x})) \int_0^1 e(\boldsymbol{\sigma}_{r,n}\cdot\boldsymbol{\beta}_n)\d\boldsymbol{\beta}_n,
\end{align*}
where the sum is over all $\mathbf{x}$ with $x_i \in I_1$ for $i = 1,\ldots,h, r+1,\ldots,r+h$ and $x_i \in I_2$ for the remaining $i$. Finally, using orthogonality and the inequality $\Lambda(x) \leq \chi_{(-1,1)}(x)$, where $\chi_S$ denotes the characteristic function of the set $S$, we have 
\[ \sum_{\mathbf{x}} \Lambda(2\xi_0(\mathbf{x})) \int_0^1 e(\boldsymbol{\sigma}_{r,n}\cdot\boldsymbol{\beta}_n)\d\boldsymbol{\beta}_n \leq \sum_{\mathbf{x}} \chi_{(-1,1)}(2\xi_0(\mathbf{x})) \int_0^1 e(\boldsymbol{\sigma}_{r,n}\cdot\boldsymbol{\beta}_n)\d\boldsymbol{\beta}_n = V_{r,h}^n(I_1,I_2;\delta). \]
This proves the required lower bound in the third estimate. 
\end{proof} 

Next, we discuss the existence of admissible Diophantine pairs. 
\begin{lemma} \label{lem:vinodiopair}
Let $k \geq 3$, and let $m \geq k(k-1)/2$. Then $(m,k)$ is admissible. 
\end{lemma}

\begin{proof} Note that the system \eqref{eq:diopairsystem} is equivalent to 
\begin{align*} 
x_1^{i}+\cdots+x_{m}^i - x_{m+1}^i-\cdots-x_{2m}^i = h_i \quad \text{ for } \quad 1 \leq i \leq k-1,
\end{align*}
where $|h_i| \leq C_iY^{i(1-1/k)}$. For fixed $\mathbf{h}=(h_1,\ldots,h_{k-1})$, let $J^0_{m,k}(Y,\mathbf{h})$ be the number of solutions to the above system with $|x_i| \leq C_0Y$. In the case $h_i = 0$ for all $i$ write $J^0_{m,k}(Y,\mathbf{h}) = J^0_{m,k}(Y)$. By orthogonality, the number of solutions to \eqref{eq:diopairsystem} is 
\[ |\mathcal{C}_{m,k}(Y)| =  \sum_{i=1}^{k-1} \, \sum_{|h_i| \leq C_iY^{i(1-1/k)}} J^0_{m,k}(Y,\mathbf{h}) \leq \sum_{i=1}^{k-1} \, \sum_{|h_i| \leq C_iY^{i(1-1/k)}} J^0_{m,k}(Y). \]
By the Main Conjecture of Vinogradov's Mean Value Theorem (see Theorem 1.1 in \cite{BDG} or Corollary 1.3 in \cite{WooNEC}), one has $J^0_{m,k}(Y) \ll Y^{2m -k(k-1)/2+\epsilon}$ for $m \geq k(k-1)/2$. Thus, 
\[ \sum_{i=1}^{k-1} \, \sum_{|h_i| \leq C_iY^{i(1-1/k)}} J^0_{m,k}(Y) \ll Y^{2m-k(k-1)/2+\epsilon}Y^{k(k-1)(1-1/k) /2} \ll Y^{2m-\frac{1}{2}(k-1)+\epsilon}, \]
and so by definition $(m,k)$ is a Diophantine pair. 
\end{proof}

We note that the above lower bound on $m$ is not sharp. One can do better even in the case $k = 3$, where the above estimate would require $m \geq 3$. We prove this in slightly more generality than is required. 
\begin{lemma} \label{lem:conj2}
Let $0 < \gamma < 1/2$ and fix positive constants $C_0, C_1, C_2$. Let $\mathcal{S}_{\gamma}(X)$ be the set of integer solutions $(y_1,y_2,y_3,y_4)$ to the system 
\begin{align*} 
& |y_1^2-y_2^2+y_3^2-y_4^2| \leq C_2 (X^{1/2}+X^{2\gamma}),
\\
& |y_1-y_2+y_3-y_4| \leq C_1X^{\gamma}, 
\end{align*}
with $|y_i| \leq C_0X^{1/2}$. Then for every $\epsilon > 0$ one has $|\mathcal{S}_{\gamma}(X)| \ll_{\epsilon} (X^{1+\gamma}+X^{1/2+3\gamma}) X^{\epsilon}$. 
\end{lemma}

Letting $\gamma = 1/3$ in the above allows us to deduce that $|\mathcal{S}_{\gamma}(X)| \ll_{\epsilon} X^{3/2+\epsilon}$, and so $(2,3)$ is an admissible Diophantine pair. 

\begin{proof} Fix $\epsilon > 0$. Let $E_1 = y_1-y_2+y_3-y_4$ and $E_2 = y_1^2-y_2^2+y_3^2-y_4^2$. It follows that
\begin{align*} 
E_2 = y_1^2-y_{2}^2 +y_3^3-(y_1-y_2+y_3-E_1)^2.
\end{align*}
Rearranging this equality, we deduce that 
\begin{align} \label{eq:product}
2(E_1+y_{2}-y_3)(y_1-y_{2})=E_2+E_1^2-2E_1y_3 
\end{align}
Note that $E_1 \ll X^{\gamma}$, $E_2 \ll X^{2\gamma}+X^{1/2}$ and $y_3 \ll X^{1/2}$. The number of possible choices of $(y_3,E_1,E_2)$, and hence $E_2+E_1^2-2E_1y_3$, is $\ll X^{1/2+\gamma}(X^{2\gamma}+X^{1/2})$. We now split into two cases, depending on whether $E_2+E_1^2-2E_1y_3$ is zero or non-zero. 

Suppose $E_2+E_1^2-2E_1y_3\neq 0$. By a divisor bound there are $\ll X^{\epsilon}$ possible values for $E_1+y_{2}-y_3$, and hence for $y_{2}$. Once $y_{2}, y_3, E_1$, and $E_2$ are determined, there is at most one possible choice for $y_1$, and the same is true for $y_{4}$. Putting these together, we indeed have $|\mathcal{S}_{\gamma}(X)| \ll X^{1/2+\gamma+\epsilon}(X^{2\gamma}+X^{1/2})$, as claimed. 

Now suppose $E_2+E_1^2-2E_1y_3 = 0$. Note that once $(y_3,E_1)$ are determined, there is at most one of $E_2$ for which this is true, and hence there are at most $\ll X^{1/2+\gamma}$ choices of $(y_1,E_1,E_2)$ for which this case applies. Using \eqref{eq:product}, it must be that either $y_1 = y_2$ or $y_2 = y_3-E_1$. Suppose first that $y_1=y_2$. Then there are at most $\ll X^{1/2}$ possible values of $(y_1,y_2)$. Also, we now have $y_4 = y_3-E_1$ and so $y_4$ is uniquely determined by $(y_3,E_1)$. Thus, $(y_1,\ldots,y_4,E_1,E_2)$ is determined by $(y_1,y_3,E_1)$ and the number of possible values of $(y_1,\ldots,y_4,E_1,E_2)$ satisfying the required bounds is $\ll X^{1+\gamma} $. If instead we have $y_2=y_3-E_1$, then $y_2$ is uniquely determined by $(y_3,E_1)$. Substituting this into the identity for $y_4$ we have $y_1 = y_4$. Hence there are at most $\ll X^{1/2}$ possible values for $(y_1,y_4)$, and therefore at most $\ll X^{1+\gamma}$ possible choices for $(y_1,\ldots,y_4,E_1,E_2)$ once again. In this case, it then follows that $|\mathcal{S}_{\gamma}(X)| \ll X^{1+\gamma}$. 
\end{proof}

\section{Mean Value Estimates} \label{sec:meanval}

The key to Theorem 3.4 in \cite{Poulias1} is to shorten the range of the variables while simultaneously achieving a sufficient power saving. Our approach to Theorems \ref{thm:main} and \ref{thm:general} extends this idea. We briefly outline the idea of the proof. By Taylor expansion and Lemma \ref{lem:poulcount}, the estimate \ref{eq:meanval} essentially reduces to the analysis of the system 
\[  \begin{cases}
      \Big|b_1 T^{\theta-1}h_1+\cdots+b_k T^{\theta-k}h_k| < 1,  \\
      \sum_{i=1}^{r} (z_i^j-z_{r+i}^j) = h_j \quad (1 \leq j \leq k),
    \end{cases} 
\] 
where $b_i$ is as defined at the beginning of Section \ref{sec:prelim}. The second line above can be handled using The Main Conjecture in Vinogradov's Mean Value Theorem, which would require $2r \geq k(k+1)$. The first line must be estimated with Lemma \ref{lem:arkzhit}. The value of $k$ in terms of $\theta$ may be reduced by shortening the ranges of the variables $z_i$, but this is done at the cost of a significant reduction of sharpness in the estimate given by Lemma \ref{lem:arkzhit}. To counteract these losses, we use an iterative process of compressing the range of the variables and extracting new information at each level, regaining the savings incrementally at each step. 

We start with a few definitions. We first decompose the interval $(X,2X]$ into essentially nested intervals of decreasing scales. Let $1 = \gamma_0 > \gamma_1 > \cdots > \gamma_n > 0$ for a fixed integer $n$. 
Given integers $q_1,\ldots,q_v$ with $0 \leq q_i \leq \lfloor X^{\gamma_{i-1}-\gamma_i} \rfloor$ for some $1 \leq v \leq n$, define the $n$-tuples 
\[ \mathbf{q}^{(v)} = (q_1,\ldots,q_v,0,\ldots,0).\]
For $1 \leq q \leq \lfloor X^{\gamma_{v-1}-\gamma_v} \rfloor$, we also write $\mathbf{q}^{(v)}|q = (q_1,\ldots, q_{v-1},q,0\ldots,0) $. Define 
\[ X_{\mathbf{q}^{(v)}} = X +\sum_{i=1}^v q_iX^{\gamma_i} = X +\sum_{i=1}^n q_iX^{\gamma_i}, \]
where $q_i = 0$ for $v+1 \leq i \leq n$. 
Given $\mathbf{q}^{(v)} = (q_1,\ldots,q_v,0,\ldots,0)$, define the intervals 
\[I(\mathbf{q}^{(v)}) = (X_{\mathbf{q}^{(v)}},X_{\mathbf{q}^{(v)}}+X^{\gamma_v}] = (X_{\mathbf{q}^{(v)}|{q_v}},X_{\mathbf{q}^{(v)}|(q_{v}+1)}].\] 
For $1 \leq v < n$ and $\mathbf{q}^{(v+1)}|q = (q_1,\ldots,q_v,q,0,\ldots,0)$, the intervals $I(\mathbf{q}^{(v+1)}|q)$ are pairwise disjoint for $q = 0, \ldots, \lfloor X^{\gamma_{h}-\gamma_{v+1}} \rfloor$, and one sees that 
\[ I({\mathbf{q}^{(v)}}) \subseteq \bigsqcup_{q=0}^{\lfloor X^{\gamma_{v}-\gamma_{v+1}} \rfloor} I({\mathbf{q}^{(v+1)}|q}). \]

We will treat the case $k = 3$ in Theorem \ref{thm:general} separately, as it serves as the base case in the general proof and also illustrates the main idea of the method. 
In the next lemma, let $\gamma_1 = 1/2$ and $\gamma_2 = 1/3$. Thus, the definitions of $\mathbf{q}^{(v)}$, $X_{\mathbf{q}^{(v)}}$ and $I(\mathbf{q}^{(v)})$ are as above for $v = 1, 2$. 

\begin{lemma} \label{lem:meancase2}
Let $\theta> 3$ and $\kappa \geq 1$. Let $m_3 \in \mathbb{N}$ such that $(m_3,3)$ is a Diophantine pair. There exists an interval $I_0 \subseteq (X,2X]$ of the form $I_0 = I(\mathbf{q}^{(2)}) $ such that 
\begin{align*} \int_{-\kappa}^{\kappa} |g(\alpha)|^{2r} \d\alpha 
& \ll \kappa X^{\frac{4r}{3}+\frac{2m_3}{3}-1+\epsilon})\int_{-1}^{1} \int_0^1 |H(\alpha,\boldsymbol{\beta}_2;I_0)|^{2r-2m_3} \d \boldsymbol{\beta}_2 \d \alpha.  
\end{align*}
\end{lemma} 

\begin{proof} By Lemma \ref{lem:poulcount} it follows with $I_1=I = (X,2X]$ and $h = r$ that 
\[ \int_{-\kappa}^{\kappa} |g(\alpha)|^{2r} \d\alpha \ll 2\kappa V_{r}(I; (2\kappa)^{-1}), \]
where $V_{r}(I; (2\kappa)^{-1})$ is the number of solutions $(x_1,\ldots,x_{2r})\in I^{2r}$ to the inequality 
\[ |x_1^{\theta}+\cdots+x_r^{\theta}-x_{r+1}^{\theta}-\cdots-x_{2r}^{\theta}|<1/(2\kappa). \]
Since $\kappa \geq 1$, it follows that $V_{r}(I; (2\kappa)^{-1}) \leq V_{r}(I; 1/2)$. It suffices to bound this latter quantity. 

By definition, we have that $\mathbf{q}^{1}|q = (q,0,\ldots,0)$ for $0 \leq q \leq \lfloor X^{1/2}\rfloor$, and that 
\[ I(\mathbf{q}^{(1)}|q) = I(\mathbf{q}^{(1)}) = (X_{\mathbf{q}^{(1)}}, X_{\mathbf{q}^{(1)}}+X^{1/2}] = (X+qX^{1/2},X+(q+1)X^{1/2}]. \]
One then has 
\[ (X,2X] \subseteq \bigsqcup_{q=0}^{\lfloor X^{1/2} \rfloor} I({\mathbf{q}^{(1)}|q}). \]
Following the first half of the proof of Theorem 3.4 in \cite{Poulias1}, referring in particular to Equations 3.18 and 3.19 in \cite{Poulias1}, we may deduce that there exists $1 \leq q_1 \leq \lfloor X^{1/2}\rfloor$
such that 
\[ V_{r}(I; 1/2) \ll X^{r-1/2} \int_{-1}^{1} |g(\alpha; I(\mathbf{q}^{(1)}|q_1))|^{2r} \d\alpha \ll X^{r-1/2} V_r(I(\mathbf{q}^{(1)}|q_1),1/2), \]
where $V_r(I(\mathbf{q}^{(1)}|q_1),1/2)$ is the number of solutions $(x_1,\ldots,x_{2r}) \in I(\mathbf{q}^{(1)}|q_1)^{2r}$ to the inequality 
\[ |x_1^{\theta}+\cdots+x_r^{\theta}-x_{r+1}^{\theta}-\cdots-x_{2r}^{\theta}|<1/2. \]

Observe that for such a solution $(x_1,\ldots,x_{2r}) \in I(\mathbf{q}^{(1)}|q_1)^{2r}$ counted in $V_r(I;1/2)$, all of the $x_i$ lie in an interval of length $\ll X^{1/2}$, namely $(X_{\mathbf{q}^{(1)}}, X_{\mathbf{q}^{(1)}}+X^{1/2}]$. Letting $y_i = x_i - X_{\mathbf{q}^{(1)}}$, one therefore has $y_i \ll X^{1/2}$. Note also that $X \leq X_{\mathbf{q}^{(1)}} \leq 2X$. By Taylor expansion, one then has 
\begin{align*} 
1/2 > \Big| \sum_{i=1}^r (x_i^{\theta}-x_{i+r}^{\theta})\Big| & = | \theta X_{\mathbf{q}^{(1)}}^{\theta-1} \sum_{i=1}^r(y_i-y_{i+r})+ O(X^{\theta-2} \sum_{i=1}^{2r} |y_i|^2)| 
\\ & = | \theta X_{\mathbf{q}^{(1)}}^{\theta-1} \sum_{i=1}^r(y_i-y_{i+r})+ O(X^{\theta-1})|,
\end{align*}
and hence 
\[ |y_1+\cdots+y_r-y_{r+1}-\cdots-y_{2r}| \ll 1.\]
Thus, $|x_1+\cdots-x_{2r}| \ll 1$. Recalling the definition of $V_r^n(I;\delta,m)$ and using the second estimate in Lemma \ref{lem:poulcount} with $I_1 = I_2 = I(\mathbf{q}^{(1)}|q_1)$, one has 
\begin{align*} 
V_r(I(\mathbf{q}^{(1)}|q_1);1/2) & = \sum_{m \ll 1} V_{r}^1(I(\mathbf{q}^{(1)}|q_1); 1/2,m) 
\\ 
& \ll \sum_{m \ll 1} \int_{-1}^{1} \int_0^1 |H(\alpha,\boldsymbol{\beta}_1; I(\mathbf{q}^{(1)}|q_1) )|^{2r} \d \boldsymbol{\beta}_1 \d \alpha 
\\ 
& \ll \int_{-1}^{1} \int_0^1 | H(\alpha,\boldsymbol{\beta}_1;I(\mathbf{q}^{(1)}|q_1) )|^{2r} \d \boldsymbol{\beta}_1 \d \alpha \ll V_{r}^1( I(\mathbf{q}^{(1)}|q_1);1/2). 
\end{align*}

We repeat this process. For $0 \leq q \leq \lfloor X^{1/6}\rfloor$, let $\mathbf{q}^{(2)}|q = (q_1,q,0,\ldots,0)$. As before, we have 
\[ I(\mathbf{q}^{(1)}|q_1) \subseteq I^* := \bigsqcup_{q=0}^{\lfloor X^{1/6}\rfloor} I(\mathbf{q}^{(2)}|q). \]
Letting $I_1 = I(\mathbf{q}^{(1)}|q_1)$ and $I_2 = I^*$ and recalling the definition of $ V_{r,h}^n(I_1, I_2; \delta)$, one has 
\begin{align*} V_{r}^1( I(\mathbf{q}^{(1)}|q_1);1/2) & \leq V_{r,m_3}^1(I_1, I_2; 1/2) 
\\ 
& \ll \int_{-1}^1 \int_0^1 | H(\alpha,\boldsymbol{\beta}_1;I(\mathbf{q}^{(1)}|q_1) )|^{2m_3} | H(\alpha,\boldsymbol{\beta}_1;I^* )|^{2r-2m_3}\d\boldsymbol{\beta}_1\d\alpha, 
\end{align*}
where in the last estimate we used Lemma \ref{lem:poulcount} with $h = m_3$ and $n = 1$. 
Next, we have 
\begin{align*} | H(\alpha,\boldsymbol{\beta}_1;I^* )|^{2r-2m_3} & = \Big|\sum_{q=0}^{\lfloor X^{1/6} \rfloor} H(\alpha,\boldsymbol{\beta}_1; I(\mathbf{q}^{(2)}|q)) \Big|^{2r-2m_3} 
\\ 
& \ll X^{r/3-m_3/3} |H(\alpha,\boldsymbol{\beta}_1; I(\mathbf{q}^{(2)}|q_2))|^{2r-2m_3} 
\end{align*}
for some $0 \leq q_2 \leq \lfloor X^{1/6} \rfloor$. Inserting this into the integral estimate for $V_{r,m_3}^1(I_1, I_2; 1/2)$ and applying Lemma \ref{lem:poulcount} we have 
\begin{align*}
V_{r,m_3}^1(I_1, I_2; 1/2) & \ll X^{\frac{r}{3}-\frac{m_3}{3}} \int_{-1}^1 \int_0^1 | H(\alpha,\boldsymbol{\beta}_1;I(\mathbf{q}^{(1)}|q_1) )|^{2m_3} |H(\alpha,\boldsymbol{\beta}_1; I(\mathbf{q}^{(2)}|q_2))|^{2r-2m_3}\d\boldsymbol{\beta}_1\d\alpha 
\\ 
& \asymp X^{\frac{r}{3}-\frac{m_3}{3}} V_{r,m_3}^1(I(\mathbf{q}^{(1)}|q_1), I(\mathbf{q}^{(2)}|q_2);1/2). 
\end{align*}

By definition, $V_{r,m_3}^1(I(\mathbf{q}^{(1)}|q_1), I(\mathbf{q}^{(2)}|q_2);1/2)$ is the number of solutions to the system 
\begin{equation} \label{eq:extradeg1}
\begin{aligned}
& |x_1^{\theta}+\cdots+x_r^{\theta}-x_{r+1}^{\theta}-\cdots-x_{2r}^{\theta}| < 1/2, 
\\ 
& \sum_{i=1}^r (x_i-x_{i+r}) = 0, 
\end{aligned}
\end{equation}
with 
\[x_1,\ldots,x_{m_3}, x_{r+1}, \ldots,x_{r+m_3} \in I(\mathbf{q}^{(1)}|q_1)\] 
and $x_i \in I(\mathbf{q}^{(2)}|q_2);1/2)$ for all other remaining $i$. Let $T_2 = \lfloor X_{\mathbf{q}^{(2)}|q_2} \rfloor$ 
and let $z_i = x_i-T_2$ for $i = 1, \ldots, 2r$. Since 
\[x_i \in I(\mathbf{q}^{(2)}|q_2);1/2) = (X_{\mathbf{q}^{(2)}|q_2}, X_{\mathbf{q}^{(2)}|q_2}+X^{1/3}]\]
for $i = 3,\ldots,r,r+3,\ldots,2r$, it follows that $1 \leq z_i \leq X^{1/3}+1$ for these $i$. Moreover, we have 
\[x_i \in I(\mathbf{q}^{(1)}|q_1) = (X_{\mathbf{q}^{(1)}|q_1}, X_{\mathbf{q}^{(1)}|q_1}+X^{1/2}] \quad \text{ for } \quad i = 1,\ldots,m_3,r+1, \ldots, r+m_3 \] 
and 
\[X_{\mathbf{q}^{(2)}|q_2} \in [X_{\mathbf{q}^{(1)}|q_1}, X_{\mathbf{q}^{(1)}|q_1}+X^{1/2}]. \] 
Therefore, $|z_i| \leq X^{1/2}+1$ for these values of $i$. Using the linear equation in \eqref{eq:extradeg1}, we then have 
\[ \sum_{i=1}^{m_3} (z_i-z_{r+i}) = \sum_{i=m_3+1}^r (z_i-z_{r+i}) \ll X^{1/3}. \]
Applying Taylor expansion to the first line of \eqref{eq:extradeg1}, it follows that 
\begin{align*} 
1/2 &  > \Big| \theta T_2^{\theta-1}\sum_{i=1}^r (z_i-z_{i+r})+\binom{\theta}{2} T_2^{\theta-2} \sum_{i=1}^r(z_i^2-z_{i+r}^2) +O(T_2^{\theta-3} \sum_{i=1}^{2r} |z_i|^3) \Big| 
\\ 
& = \Big| \binom{\theta}{2} T_2^{\theta-2} \sum_{i=1}^r(z_i^2-z_{i+r}^2) +O(T_2^{\theta-3/2}) \Big| 
= \Big| \binom{\theta}{2} T_2^{\theta-2} \sum_{i=1}^{m_3} (z_i^2-z_{i+r}^2) +O(T_2^{\theta-4/3}) \Big|, 
\end{align*}
where in the first equality we used $z_1+\cdots-z_{2r} = 0$ and $|z_i| \ll X^{1/2}$ for all $i$, while in the second we used $|z_i| \ll X^{1/3}$ for $i \neq 1,\ldots,m_3,r+1, \ldots,r+m_3$. Rearranging gives us 
\[ |\sum_{i=1}^{m_3} (z_i^2-z_{r+i}^2)| \ll T_2^{2/3} \ll X^{2/3}. \]
Thus, there are constants $C_1, C_2$ independent of $X$ such that $(z_1,\ldots,z_{r+m_3},z_{r+1},\ldots, z_{r+m_3})$ is a solution to the system 
\begin{align*}
& |\sum_{i=1}^{m_3} (y_i^2-y_{m_3+i}^2)| \leq C_2X^{2/3}, 
\\ 
& |\sum_{i=1}^{m_3} (y_i-y_{m_3+i})| \leq C_1X^{1/3}. 
\end{align*}
Let $\mathcal{S}_3$ be the set of tuples $(x_1,\ldots, x_{m_3},x_{r+1},\ldots, x_{r+m_3}) \in I(\mathbf{q}^{(1)}|q_1)^{2m_3}$ such that the variables $z_i = x_i-T_2$ are solutions to the above. By Lemma \ref{lem:poulcount} we thus conclude 
\begin{align*} 
V_{r,m_3}^1(I(\mathbf{q}^{(1)}|q_1), I(\mathbf{q}^{(2)}|q_2);1/2) & = V_{r,m_3}^1(\mathcal{S}_3, I(\mathbf{q}^{(2)}|q_2);1/2)
\\ 
& \ll \int_{-1}^{1} \int_0^1 |H_{\mathcal{S}_3}(\alpha,\boldsymbol{\beta}_1)H(\alpha,\boldsymbol{\beta}_1;I(\mathbf{q}^{(2)}|q_2))^{2r-2m_3}| \d \boldsymbol{\beta}_1 \d \alpha.
\end{align*}
Since $(m_3,3)$ is a Diophantine pair with $|z_i| \ll X^{1/2}$, we have the estimate $ |\mathcal{S}_3| \ll X^{m_3-1/2+\epsilon}$. Hence, 
\begin{align*} 
\int_{-1}^{1} \int_0^1 |H_{\mathcal{S}_2}(\alpha,\boldsymbol{\beta}_1) & H(\alpha,\boldsymbol{\beta}_1;I(\mathbf{q}^{(2)}|q_2))^{2r-2m_3}| \d \boldsymbol{\beta}_1 \d \alpha 
\\ 
& \ll X^{m_3-1/2+\epsilon} \int_{-1}^{1} \int_0^1 |H(\alpha,\boldsymbol{\beta}_1;I(\mathbf{q}^{(2)}|q_2))|^{2r-2m_3} \d \boldsymbol{\beta}_1 \d \alpha. 
\end{align*}
Collecting all estimates together, we deduce that 
\begin{align*}
V_r(I;1/2) \ll X^{ 4r/3+2m_3/3-1+\epsilon}\int_{-1}^{1} \int_0^1 |H(\alpha,\boldsymbol{\beta}_1;I(\mathbf{q}^{(2)}|q_2))|^{2r-2m_3} \d \boldsymbol{\beta}_1 \d \alpha. 
\end{align*}
Applying Lemma \ref{lem:poulcount} once more, we have 
\begin{align*} 
\int_{-1}^{1} \int_0^1 |H(\alpha,\boldsymbol{\beta}_1;I(\mathbf{q}^{(2)}|q_2))|^{2r-2m_3} \d \boldsymbol{\beta}_1 \d \alpha \ll V_{r-m_3}^1(I(\mathbf{q}^{(2)}|q_2);1/2), 
\end{align*}
where we recall that the latter quantity is the number of integer solutions to the system 
\begin{equation} \label{eq:deg1minus4}
\begin{aligned}
& |x_1^{\theta}+\cdots+x_{r-m_3}^{\theta}-x_{r-m_3+1}^{\theta}-\cdots-x_{2r-2m_3}^{\theta}| < 1/2, 
\\ 
& \sum_{i=1}^{r-m_3} (x_i-x_{r-m_3+i}) = 0, 
\end{aligned}
\end{equation}
with $(x_1,\ldots,x_{2r-2m_3}) \in I(\mathbf{q}^{(2)}|q_2)^{2r-2m_3}$. Now, $z_i = x_i -T_2$ satisfies $1 \leq z_i \leq X^{1/3}+1$ for each $i$. Applying Taylor expansion to the first line of \eqref{eq:deg1minus4}, as before, then shows that 
\begin{align*} 
1/2 &  > \Big| \theta T_2^{\theta-1}\sum_{i=1}^{r-m_3} (z_i-z_{i+r-m_3})+\binom{\theta}{2} T_2^{\theta-2} \sum_{i=1}^{r-m_3} (z_i^2-z_{i+r-m_3}^2) +O\Big(T_2^{\theta-3} \sum_{i=1}^{2r-2m_3} |z_i|^3\Big) \Big| 
\\ 
& = \Big| \binom{\theta}{2} T_2^{\theta-2} \sum_{i=1}^{r-m_3} (z_i^2-z_{i+r-m_3}^2) +O(T_2^{\theta-2}) \Big|,  
\end{align*}
and hence 
\[ |z_1^2+\cdots+z_{r-m_3}^2-z_{r-m_3+1}^2 -\cdots-z_{2r-2m_3}^2| \ll 1.\]
Using the second equation in \eqref{eq:deg1minus4}, this then implies that 
\[|x_1^2+\cdots+x_{r-m_3}^2-x_{r-m_3+1}^2 -\cdots-x_{2r-2m_3}^2| \ll 1.\] Thus, 
\begin{align*}
V_{r-m_3}^1(I(\mathbf{q}^{(2)}|q_2);1/2) & = \sum_{m \ll 1} V_{r-m_3}^2(I(\mathbf{q}^{(2)}|q_2); 1/2,m)
\\ 
& \ll \sum_{m \ll 1} \int_{-1}^{1} \int_0^1 |H(\alpha,\boldsymbol{\beta}_2; I(\mathbf{q}^{(2)}|q_2) )|^{2r-2m_3} \d \boldsymbol{\beta}_2 \d \alpha 
\\ 
& \ll \int_{-1}^{1} \int_0^1 | H(\alpha,\boldsymbol{\beta}_2;I(\mathbf{q}^{(2)}|q_2) )|^{2r-2m_3} \d \boldsymbol{\beta}_2 \d \alpha \ll V_{r-m_3}^2( I(\mathbf{q}^{(2)}|q_2);1/2). 
\end{align*}
The claimed estimate follows from this. 
\end{proof}

We may iterate this procedure further. For $v \in \mathbb{N}$, let $\gamma_v = 1/(v+1)$. As before, the definitions of $\mathbf{q}^{(v)}$, $X_{\mathbf{q}^{(v)}}$ and $I(\mathbf{q}^{(v)})$ follow as constructed above. Note that 
\[ |I(\mathbf{q}^{(v)})| \asymp X^{\gamma_v} = X^{\frac{1}{v+1}}.\] 

\begin{lemma} \label{lem:meancasegen}
Let $\theta > 3$ be a non-integer and let $3 \leq n \leq \lfloor \theta \rfloor$ be an integer. Let $(m_k,k)$ be a sequence of Diophantine pairs for $3 \leq k \leq n$ and let $M_k = m_3+\cdots+m_k$. For $X, \kappa \geq 1$, there exists an interval $I_0 \subseteq (X,2X]$ of the form $I_0 = I(\mathbf{q}^{(n-1)})$ such that for every $\epsilon > 0$ one has 
\[ \int_{-\kappa}^{\kappa} |g(\alpha)|^{2r} \d\alpha \ll_{\epsilon} \kappa X^{2r(1-\frac{1}{n})+\frac{2M_n}{n}-\frac{n-1}{2}+\epsilon} \int_{-1}^{1} \int_{0}^{1} |H(\alpha,\boldsymbol{\beta}_{n-1};I_0 )|^{2r-2M_n} \d \boldsymbol{\beta}_{n-1} \d \alpha. \]
\end{lemma}

\begin{proof} We proceed by induction on $n$. By Lemma \ref{lem:meancase2} the case $n=3$ follows. Indeed, we have
\[ \int_{-\kappa}^{\kappa} |g(\alpha)|^{2r} \d\alpha \ll \kappa X^{2r(1-\frac{1}{3})+\frac{2m_3}{3}-\frac{3-1}{2}+\epsilon} \int_{-1}^1 \int_0^1 |H(\alpha,\boldsymbol{\beta}_{n-1};I_0 )|^{2r-2m_3} \d \boldsymbol{\beta}_{2} \d \alpha. \]
Therefore, we assume $\theta > 4$ and let $3 < n \leq \lfloor \theta \rfloor$ such that Lemma \ref{lem:meancasegen} holds for $n-1$. Then there exists an interval $I \subseteq (X,2X]$ such that $I = I(\mathbf{q}^{(n-2)})$ and 
\[ \int_{-\kappa}^{\kappa} |g(\alpha)|^{2r} \d\alpha \ll \kappa X^{2r(1-\frac{1}{n-1})+\frac{2M_{n-1}}{n-1}-\frac{n-2}{2}+\epsilon} \int_{-1}^{1} \int_{0}^{1} |H(\alpha,\boldsymbol{\beta}_{n-2};I )|^{2r-2M_{n-1}} \d \boldsymbol{\beta}_{n-2} \d \alpha. \]
Applying Lemma \ref{lem:poulcount} we have 
\begin{align*} 
\int_{-1}^{1} \int_{0}^{1} |H(\alpha,\boldsymbol{\beta}_{n-2};I )|^{2r-2M_{n-1}} \d \boldsymbol{\beta}_{n-2} \d \alpha \ll V_{r-M_{n-1}}^{n-2}(I;1/2).
\end{align*}
Let $r_0 = r-M_{n-1}$. The right hand side is the number of integer solutions $(x_1,\ldots, x_{2r_0}) \in I^{2r_0}$ to the system 
\begin{equation} \label{eq:degminustwo}
\begin{aligned}
|x_1^{\theta}+\cdots+x_{r_0}^{\theta}-x_{r_0+1}^{\theta}-\cdots-x_{r_0+1}^{\theta}| < 1/2, 
\\
\sum_{i=1}^{r_0} (x_i^j-x_{r_0+i}^j) = 0 \quad \text{ for } \quad 1 \leq j \leq n-2.
\end{aligned}
\end{equation}
As observed earlier, 
\[ I = I(\mathbf{q}^{(n-2)}) \subseteq \bigsqcup_{q=0}^{\lfloor X^{\frac{1}{n(n-1)}} \rfloor} I({\mathbf{q}^{(n-1)}|q}). \]
Define $I^*$ to be the interval that is the disjoint union on the right. Therefore, it follows that 
\[ V_{r_0}^{n-2}(I;1/2) \ll V_{r_0,m_n}^{n-2}(I, I^* ;1/2), \]
where we recall that the quantity on the right is the number of integer solutions $(x_1,\ldots, x_{2r_0}) $ to the system above with
\[ (x_1,\ldots, x_{m_n}, x_{r_0+1},\ldots,x_{r_0+m_n}) \in I^{2m_n} \] 
and $x_i \in I^*$ for all other $i$. 

Applying Lemma \ref{lem:poulcount}, we have 
\[ V_{r_0,m_n}^{n-2}(I, I^* ;1/2) \ll \int_{-1}^{1} \int_{0}^{1} |H(\alpha,\boldsymbol{\beta}_{n-2};I )|^{2m_n} |H(\alpha,\boldsymbol{\beta}_{n-2};I^* )|^{2r_0-2m_n} \d \boldsymbol{\beta}_{n-2} \d \alpha. \]
By definition of $I^*$, we use the Triangle Inequality to deduce that 
\begin{align*} |H(\alpha,\boldsymbol{\beta}_{n-2};I^* )|^{2r_0-2m_n} & \leq \Big( \sum_{q = 1}^{\lfloor X^{\frac{1}{n(n-1)}} \rfloor} |H(\alpha,\boldsymbol{\beta}_{n-2};I({\mathbf{q}^{(n-1)}|q}) )|  \Big)^{2r_0-2m_n}
\\ 
& \ll X^{\frac{2r_0}{n(n-1)}-\frac{2m_n}{n(n-1)}} |H(\alpha,\boldsymbol{\beta}_{n-2};I({\mathbf{q}^{(n-1)}|q_{n-1}}) )|^{2r_0-2m_n}
\end{align*}
for some $1 \leq q_{n-1} \leq \lfloor X^{\frac{1}{n(n-1)}} \rfloor$. Using this inequality and Lemma \ref{lem:poulcount} yet again, one has 
\begin{align*} 
\int_{-1}^{1} & \int_{0}^{1} |H(\alpha,\boldsymbol{\beta}_{n-2};I )|^{2m_n} |H(\alpha,\boldsymbol{\beta}_{n-2};I^* )|^{2r_0-2m_n} \d \boldsymbol{\beta}_{n-2} \d \alpha.
\\ 
& \ll X^{\frac{2r_0}{n(n-1)}-\frac{2m_n}{n(n-1)}} \int_{-1}^{1} \int_{0}^{1} |H(\alpha,\boldsymbol{\beta}_{n-2};I )|^{2m_n} |H(\alpha,\boldsymbol{\beta}_{n-2};I({\mathbf{q}^{(n-1)}|q_{n-1}}) )|^{2r_0-2m_n} \d \boldsymbol{\beta}_{n-2} \d \alpha 
\\ 
& \ll X^{\frac{2r-2M_{n-1}}{n(n-1)}-\frac{2m_n}{n(n-1)}} V_{r_0,m_n}^{n-2}(I,I({\mathbf{q}^{(n-1)}|q_{n-1}}),1/2).
\end{align*}
Combining this estimate with the inductive assumption we have 
\begin{align*} 
\int_{-\kappa}^{\kappa} |g(\alpha)|^{2r} \d\alpha \ll \kappa X^{2r(1-\frac{1}{n})+\frac{2M_{n-1}}{n}-\frac{n-2}{2}-\frac{2m_{n}}{n(n-1)}+\epsilon} V_{r_0,m_n}^{n-2}(I,I({\mathbf{q}^{(n-1)}|q_{n-1}}),1/2). 
\end{align*}
It therefore suffices to show for any $\epsilon > 0$ that 
\begin{align*} V_{r_0,m_n}^{n-2} & (I,I({\mathbf{q}^{(n-1)}|q_{n-1}}),1/2) 
\\ 
& \ll X^{\frac{2m_n}{n-1}-\frac{1}{2}+\epsilon} \int_{-1}^{1} \int_{0}^{1} |H(\alpha,\boldsymbol{\beta}_{n-1};I({\mathbf{q}^{(n-1)}|q_{n-1}}) )|^{2r-2m_n} \d \boldsymbol{\beta}_{n-1} \d \alpha. 
\end{align*}
By definition, $V_{r_0,m_n}^{n-2}(I,I({\mathbf{q}^{(n)}|q_{n-1}}),1/2)$ counts the integer solutions $(x_1,\ldots,x_{2r_0})$ to \eqref{eq:degminustwo} with 
\[ x_1,\ldots,x_{m_n},x_{r_0+1},\ldots,x_{r_0+m_n} \in I = (X_{\mathbf{q}^{(n-2)}}, X_{\mathbf{q}^{(n-2)}}+X^{1/(n-1)}]\] 
and 
\[x_i \in I({\mathbf{q}^{(n-1)}|q_{n-1}}) = (X_{\mathbf{q}^{(n-1)}|q_{n-1}}, X_{\mathbf{q}^{(n-1)}|q_{n-1}}+X^{1/n}] \] 
for all remaining $i$. Write $T_{n-1} = \lfloor X_{\mathbf{q}^{(n-1)}|q_{n-1}} \rfloor $ and $z_i = x_i-T_{n-1}$. Then 
\[ 1 \leq z_i \leq X^{1/n}+1 \quad \text{ for } \quad i = m_n+1,\ldots,r_0,r_0+m_n+1,\ldots,2r_0. \] Moreover, since 
\[ X_{\mathbf{q}^{(n-1)}|q_{n-1}} \in [X_{\mathbf{q}^{(n-2)}},X_{\mathbf{q}^{(n-2)}}+X^{1/(n-1)}],\] 
it follows that $|z_i| \leq X^{1/(n-1)}+1$ for $i = 1,\ldots,m_n,r_0+1,\ldots,r_0+m_n$. Taylor expansion of the top line in \eqref{eq:degminustwo} allows us to deduce then that 
\begin{align*} 
1/2 & > \Big|\sum_{j=1}^{n-2} \binom{\theta}{j} T_{n-1}^{\theta-j} \sum_{i=1}^{r_0} (z_i^j-z_{r_0+i}^j) + \binom{\theta}{n-1}T_{n-1}^{\theta-n+1} \sum_{i=1}^{r_0} (z_i^{n-1} -z_{r_0+i}^{n-1}) + O\Big(X^{\theta-n}\sum_{i=1}^{2r_0} |z_i|^n\Big) \Big| 
\\ 
& = \Big| \binom{\theta}{n-1}T_{n-1}^{\theta-n+1} \sum_{i=1}^{r_0} (z_i^{n-1} -z_{r_0+i}^{n-1}) + O(X^{\theta-n+\frac{n}{n-1}}) \Big|
\\ 
& = \Big| \binom{\theta}{n-1}T_{n-1}^{\theta-n+1} \sum_{i=1}^{m_n} (z_i^{n-1} -z_{r_0+i}^{n-1}) +O(X^{\theta-n+1+\frac{n-1}{n}}) + O(X^{\theta-n+\frac{n}{n-1}}) \Big|,
\end{align*}
where in the last equality we used the bound $|z_i|, |z_{r_0+i|} \ll X^{1/n}$ for $i \neq 1,\ldots,m_n$. Hence, 
\[ |z_1^{n-1}+\cdots+z_{m_n}^{n-1}-z_{r_0+1}^{n-1}-\cdots-z_{r_0+m_n}^{n-1}| \ll X^{\frac{n-1}{n}}. \] 
Since the system of equations described by the second line in \eqref{eq:degminustwo} is translation invariant in the sense that $(a_1,\ldots,a_{2r_0})$ is a solution if and only if $(a_1+b,\ldots, a_{2r_0}+b)$ is, we have that $(z_1,\ldots,z_{2r_0})$ is a solution to that system as well. Thus, 
\[ \sum_{i=1}^{r_0} (z_i^j-z_{r_0+i}^j) = 0 \quad \text{ for } \quad 1 \leq j \leq n-2, \]
and therefore 
\[ \sum_{i=1}^{m_n} (z_i^j-z_{r_0+i}^j) = -\sum_{i=m_n+1}^{r_0} (z_i^j-z_{r_0+i}^j) \ll X^{j/n} \quad \text{ for } \quad 1 \leq j \leq n-2. \]
Hence, there exist constants $C_1,\ldots, C_{n-1}$ such that 
\[ \Big| \sum_{i=1}^{m_n} (z_i^j - z_{r_0+i}^j ) \Big| \leq C_j X^{\frac{j}{n}} \quad \text{ for } \quad 1 \leq j \leq n-1. \]
Let $\mathcal{S}_{n}$ be the set of tuples $(x_1,\ldots,x_{r_0+m_n}) \in I^{2m_n}$ such that the variables $z_i=x_i-T_{n-1}$ are solutions to the above system. Since $(m_n,n)$ is a Diophantine pair, we have $|\mathcal{S}_{n}| \ll X^{\frac{2m_n}{n-1}-\frac{1}{2}+\epsilon}$. Here, note that we have used $Y = X^{\frac{1}{n-1}}$. Applying Lemma \ref{lem:poulcount}, it follows that 
\begin{align*} 
V_{r_0,m_n}^{n-2} & (I,I({\mathbf{q}^{(n-1)}|q_{n-1}}),1/2) 
\\ 
&\ll \int_{-1}^{1} \int_{0}^{1} |H_{\mathcal{S}_2}(\alpha, \boldsymbol{\beta}_{n-2})| |H(\alpha,\boldsymbol{\beta}_{n-2};I({\mathbf{q}^{(n-1)}|q_{n-1}}) )|^{2r_0-2m_n} \d \boldsymbol{\beta}_{n-2} \d \alpha 
\\ 
& \ll X^{\frac{2m_n}{n-1}-\frac{1}{2}+\epsilon} \int_{-1}^{1} \int_{0}^{1} |H(\alpha,\boldsymbol{\beta}_{n-2};I({\mathbf{q}^{(n-1)}|q_{n-1}}) )|^{2r_0-2m_n} \d \boldsymbol{\beta}_{n-2} \d \alpha. 
\\ 
& \ll X^{\frac{2m_n}{n-1}-\frac{1}{2}+\epsilon} V_{r_0-m_n}^{n-2}(I({\mathbf{q}^{(n-1)}|q_{n-1}});1/2). 
\end{align*}
Now we have $z_i = x_i-T_{n-1} \in [1,X^{1/n}+1]$ for each $i$, and again these variables satisfy \eqref{eq:degminustwo} with $r_0$ replaced with $r_1 = r_0-m_n=r-M_n$. Taylor expanding again, we therefore deduce that 
\begin{align*} 
1/2 & > \Big| \sum_{j=1}^{n-2} \binom{\theta}{j} T_{n-1}^{\theta-j} \sum_{i=1}^{r_1} (z_i^j-z_{r_1+i}^j) + \binom{\theta}{n-1}T_{n-1}^{\theta-n+1} \sum_{i=1}^{r_1} (z_i^{n-1} -z_{r_1+i}^{n-1}) + O\Big(X^{\theta-n}\sum_{i=1}^{2r_1} |z_i|^n\Big) \Big|
\\ 
& = \Big| \binom{\theta}{n-1}T_{n-1}^{\theta-n+1} \sum_{i=1}^{r_1} (z_i^{n-1} -z_{r_1+i}^{n-1}) + O(X^{\theta-n+1}) \Big|,
\end{align*}
and hence 
\[ \sum_{i=1}^{r_1} (z_i^{n-1} -z_{r_1+i}^{n-1}) \ll 1. \]
This shows, after applying Lemma \ref{lem:poulcount}, that 
\begin{align*}
V_{r_1}^{n-2}(I({\mathbf{q}^{(n-1)}|q_{n-1}});1/2) & = \sum_{m \ll 1} V_{r_1}^{n-1}(I({\mathbf{q}^{(n-1)}|q_{n-1}});1/2, m)
\\ 
& \ll \sum_{m \ll 1} \int_{-1}^{1} \int_{0}^{1} |H(\alpha,\boldsymbol{\beta}_{n-1};I({\mathbf{q}^{(n-1)}|q_{n-1}}) )|^{2r_1} \d \boldsymbol{\beta}_{n-1} \d \alpha
\\ 
& \ll \int_{-1}^{1} \int_{0}^{1} |H(\alpha,\boldsymbol{\beta}_{n-1};I({\mathbf{q}^{(n-1)}|q_{n-1}}) )|^{2r_1} \d \boldsymbol{\beta}_{n-1} \d \alpha. 
\end{align*}
Recalling that $r_1 = r_0-m_n = r-M_n$, it follows by combining the estimates that 
\begin{align*} 
\int_{-\kappa}^{\kappa} |g(\alpha)|^{2r} \d\alpha & \ll \kappa X^{2r(1-\frac{1}{n})+\frac{2M_{n-1}}{n}-\frac{n-2}{2}-\frac{2m_{n}}{n(n-1)}+\epsilon} V_{r_0,m_n}^{n-2}(I,I({\mathbf{q}^{(n-1)}|q_{n-1}}),1/2) \\ 
& \ll \kappa X^{2r(1-\frac{1}{n})+\frac{2M_{n-1}}{n}-\frac{n-1}{2}+\frac{2m_{n}}{n}+\epsilon} V_{r_0-m_n}^{n-2}(I({\mathbf{q}^{(n-1)}|q_{n-1}});1/2)
\\ 
& \ll \kappa X^{2r(1-\frac{1}{n})+\frac{2M_{n}}{n}-\frac{n-1}{2}+\epsilon} \int_{-1}^{1} \int_{0}^{1} |H(\alpha,\boldsymbol{\beta}_{n-1};I({\mathbf{q}^{(n-1)}|q_{n-1}}) )|^{2r-2M_n} \d \boldsymbol{\beta}_{n-1} \d \alpha, 
\end{align*}
which completes the proof. 
\end{proof}

From here, a familiar Taylor series argument from \cite{ArkhZhit} and \cite{Poulias1} completes the argument. 
We record this as a lemma. 

\begin{lemma} \label{lem:taylor}
Let $\theta > 3$. Let $I = (X_0,X_0+X_1] \subseteq [X,2X]$ be an interval with $|I| \asymp X^{1/v}$ for some integer $ 2 \leq v \leq \lfloor \theta \rfloor$, and let $1 \leq u \leq v$ be  another integer. For 
\[ 2r \geq \lceil \theta(1-1/v)^{-1} \rceil(\lceil \theta(1-1/v)^{-1} \rceil+1)+2M_v, \] 
one has 
\[ \int_{-1}^1 \int_0^1 |H(\alpha, \boldsymbol{\beta}_{u-1};I)|^{2r-2M_v} \d\boldsymbol{\beta}_{u-1}\d\alpha \ll_{\epsilon} X^{\frac{2r}{v}-\frac{2M_v}{v}+\frac{v-1}{2}-\theta+\epsilon}.  \]
\end{lemma}

\begin{proof} Write $r_0 = r-M_v$. By Lemma \ref{lem:poulcount}, one has that 
\[ \int_{-1}^1 \int_0^1 |H(\alpha, \boldsymbol{\beta}_{u-1};I)|^{2r_0} \d\boldsymbol{\beta}_{u-1}\d\alpha \ll V_{r_0}^{n-1}(I;1/2), \]
the latter quantity being the number of solutions $(x_1,\ldots,x_{2r_0}) \in I^{2r_0} $ to the system 
\begin{align*} 
& \Big| \sum_{i=1}^{r_0} (x_i^{\theta} -x_{r_0+i}^{\theta}) \Big| < 1/2, 
\\ 
& \sum_{i=1}^{r_0} (x_i^j-x_{r_0+i}^j) = 0, \quad j= 1, \ldots, u-1.
\end{align*}
The number of solutions this system is obviously bounded by the number of solutions to the first line, namely the inequality. Let $T = X_0$, and let $z_i = x_i-T$. By Taylor expansion with $k \geq \lceil \theta (1-1/v)^{-1} \rceil$ terms, we see that the Diophantine inequality above is equivalent to 
\begin{align} \label{eq:taylorineq}
\Big| \sum_{j=1}^{k} \binom{\theta}{j} T^{\theta-j} \sum_{i=1}^{r_0} (z_i^j-z_{r_0+i}^j) + \sum_{i=1}^{r_0}(R_k(z_i)-R_k(z_{r_0+i})) \Big| < 1/2,
\end{align}
where $R_k$ is the remainder from Taylor's Theorem. It is well known that 
\[ R_k(z_i) \ll T^{\theta-k-1}|z_i|^{k+1} \ll T^{\theta-k-1}X^{(k+1)/v} \ll X^{\theta-(k+1)(1-1/v)}, \]
and since $k > \theta (1-1/v)^{-1}-1$ it follows that the sum of the remainders in \eqref{eq:taylorineq} is $o(1)$. Therefore, for sufficiently large $X$ the number of solutions to \eqref{eq:taylorineq} is bounded above by 
the number of solutions $(z_1,\ldots,z_{2r_0}) \in [1,X_1+1]^{2r_0}$ to the inequality 
\[ \Big| \sum_{j=1}^{k} \binom{\theta}{j} T^{\theta-j} \sum_{i=1}^{r_0} (z_i^j-z_{r_0+i}^j) \Big| < 1, \]
which is in turn bounded by the number of solutions $(z_1,\ldots,z_{2r_0})$ to the system 
\[  \begin{cases}
      \Big|\binom{\theta}{1} T^{\theta-1}h_1+\cdots+\binom{\theta}{k} T^{\theta-k}h_k| < 1,  \\
      \sum_{i=1}^{r_0} (z_i^j-z_{r_0+i}^j) = h_j \quad (1 \leq j \leq k)
    \end{cases} 
\] 
with $1 \leq z_i \leq X_1+1$. Write $Y = X_1+1$. Write $Z_{r_0,k}(Y; \mathbf{h})$ to denote the number of solutions to this system, where $\mathbf{h} = (h_1,\cdots, h_k)$. Note that the integers $h_j$ satisfy $|h_j| \leq r_0Y^{j}$. 

Define $J_{r_0,k}(Y; \mathbf{h})$ to be the number of integer solutions to the system 
\[ \sum_{i=1}^{r_0} (z_i^j - z_{r_0+i}^j) = h_j \quad (1 \leq j \leq k), \]
with $0 < z_i \leq Y$. When $\mathbf{h} = \mathbf{0}$ we simply abbreviate this as $J_{r_0,k}(Y)$. By orthogonality it follows that $J_{r_0,k}(Y; \mathbf{h}) \leq J_{r_0,k}(Y)$ for any $\mathbf{h}$. By the Main Conjecture of Vinogradov's Mean Value Theorem, one has for $r_0 \geq k(k+1)/2$ that 
\[ J_{r_0,k}(Y; \mathbf{h}) \leq J_{r_0,k}(Y) \ll Y^{2r_0-k(k+1)/2+\epsilon}. \]

As a result, one has 
\[ Z_{r_0,k}(Y; \mathbf{h}) \ll \sum_{\substack{ |h_j| \leq r_0Y^j \\ 1 \leq j \leq k \\ |\mathcal{H}(\mathbf{h})| < 1 }} J_{r_0,k}(Y; \mathbf{h}) \ll Y^{2r_0- k(k+1)/2 + \epsilon} \sum_{\substack{ |h_j| \leq r_0Y^j \\ 1 \leq j \leq k \\ |\mathcal{H}(\mathbf{h})| < 1 }} 1, \]

where 
\[\mathcal{H}(\mathbf{h}) = \binom{\theta}{1}T^{\theta-1}h_1+\cdots+\binom{\theta}{k} T^{\theta-k}h_k. \] 
Since $X \leq T+1 \leq 2X+1$ and $|I| \asymp X^{1/v}$, we have $Y \asymp X^{1/v} \asymp T^{1/v}$. We apply Lemma \ref{lem:arkzhit} with $P = T$, $t = r_0$, $\gamma = 1/v$, and $m = k$. Noting that $v^{-1} \geq \lfloor \theta \rfloor^{-1}$, Remark \ref{rem:1} reveals that 
\[ \sum_{\substack{ |h_j| \leq r_0Y^j \\ 1 \leq j \leq k \\ |\mathcal{H}(\mathbf{h})| < 1 }} 1 \ll T^{v^{-1} k(k+1)/2+\upsilon-\theta} \ll X^{v^{-1} k(k+1)/2+\frac{v-1}{2}-\theta}. \]
Let $k = \lceil \theta(1-1/v)^{-1} \rceil$. As long as $2r_0 \geq k(k+1)$, we see that 
\[ Z_{r_0,k}(Y; \mathbf{h}) \ll Y^{2r_0- k(k+1)/2 + \epsilon}X^{v^{-1} k(k+1)/2+\frac{v-1}{2}-\theta} \ll X^{\frac{2r}{v}-\frac{2M_v}{v}+\frac{v-1}{2}-\theta+\epsilon}. \]
Hence, the claimed estimate holds for $2r \geq \lceil \theta(1-1/v)^{-1} \rceil(\lceil \theta(1-1/v)^{-1} \rceil+1)+2M_v$, as required. 
\end{proof}

The proofs of Theorems \ref{thm:main} and \ref{thm:general} are straightforward from here. 

\begin{proof}[Proof of Theorems \ref{thm:main} and \ref{thm:general}]
We begin with Theorem \ref{thm:general}. Suppose $\theta > 3$ and $3 \leq n \leq \lfloor \theta \rfloor$. Let $2r \geq \lceil \theta(1-1/n)^{-1} \rceil(\lceil \theta(1-1/n)^{-1}\rceil+1)+2M_n$. We apply Lemma \ref{lem:taylor} with $u=v = n$. By this lemma and Lemma \ref{lem:meancasegen}, we have 
\begin{align*} \int_{-\kappa}^{\kappa} |g(\alpha)|^{2r} \d\alpha & \ll \kappa X^{2r(1-\frac{1}{n})+\frac{2M_n}{n}-\frac{n-1}{2}+\epsilon} \int_{-1}^{1} \int_{0}^{1} |H(\alpha,\boldsymbol{\beta}_{n-1};I_0 )|^{2r-n(n-1)+2} \d \boldsymbol{\beta}_{n-1} \d \alpha
\\ 
& \ll \kappa X^{2r(1-\frac{1}{n})+\frac{2M_n}{n}-\frac{n-1}{2}+\epsilon} X^{\frac{2r}{n}-\frac{2M_n}{n}+\frac{n-1}{2}-\theta}
\\ 
& \ll \kappa X^{2r-\theta+\epsilon}, 
\end{align*}
proving Theorem \ref{thm:general}. 

By Lemmas \ref{lem:vinodiopair} and \ref{lem:conj2}, it follows that $(m_k,k)_{k=3}^{n}$ is a sequence of Diophantine pairs if $m_k = k(k-1)/2$ for $k \geq 4$ and $m_k = 2$ for $k = 3$. Let $n = \lfloor c \theta^{1/2} \rfloor+2$ for some positive constant $c$ to be chosen later; here we assume $\theta$ is large enough that $3 \leq \lfloor c \theta^{1/2} \rfloor +2 \leq \lfloor \theta \rfloor$. Note that if $c = 1$, then $3 \leq n \leq \lfloor \theta \rfloor$ for $\theta > 3$. It follows that 
\begin{align*} 2M_n & = \sum_{k=3}^n 2m_k=4+ \sum_{k=4}^n k(k-1) = \frac{1}{3}n(n+1)(n-1)-4
\\ 
& = \frac{1}{3}( \lfloor c \theta^{1/2}\rfloor +1) \rfloor( \lfloor c \theta^{1/2} \rfloor+2)( \lfloor c \theta^{1/2} \rfloor+3)-4
\\ 
& \leq \frac{1}{3}(c^3\theta^{3/2}+6c^2\theta+11c\theta^{1/2}+6)-4 = \frac{c^3}{3}\theta^{3/2}+2c^2\theta +\frac{11c}{3}\theta^{1/2}-2.
\end{align*}
Next, one has 
\begin{align*} 
\lceil\theta(1-1/n)^{-1}\rceil(\lceil\theta(1-1/n)^{-1}\rceil+1) & = \lceil \theta(1+(\lfloor c\theta^{1/2} \rfloor +1)^{-1} \rceil (\lceil \theta(1+(\lfloor c\theta^{1/2} \rfloor +1)^{-1} \rceil+1)
\\ 
& \leq (\theta+c^{-1}\theta^{1/2}+1)(\theta+c^{-1}\theta^{1/2}+2)
\\ 
 & = \theta^{2}+\frac{2}{c}\theta^{3/2}+(3+\frac{1}{c^2})\theta+\frac{3}{c}\theta^{1/2}+2.
\end{align*}
Therefore, if 
\[ 2r \geq \theta^2 +(2/c+c^3/3)\theta^{3/2}+(2c^2+3+1/c^2)\theta+(11c/3+3/c)\theta^{1/2}, \]
then one indeed has 
\begin{align*} 2r & \geq \lceil\theta(1-1/n)^{-1}\rceil(\lceil\theta(1-1/n)^{-1}\rceil+1) + 2M_n, 
\end{align*}
and so the estimate \eqref{eq:meanval} holds by Theorem \ref{thm:general}, as long as $3 \leq \lfloor c \theta^{1/2} \rfloor +2 \leq \lfloor \theta \rfloor$. 
Optimizing for the second order term would lead us to choose $c = 2^{1/4}$. Note that if $\theta > 4$ then indeed $3 \leq \lfloor c \theta^{1/2} \rfloor +2 \leq \lfloor \theta \rfloor$, and so in this case we let $c = 2^{1/4}$. If $3 < \theta < 4$ then we set $c = 1$. Thus, estimate \eqref{eq:meanval} holds for 
\[ 2r \geq \theta^2 +7\theta^{3/2}/3+6\theta+20\theta^{1/2}/3 \]
if $\theta > 3$ and for 
\[ 2r \geq \theta^2+\frac{2^{11/4}}{3} \theta^{3/2}+(3+\frac{5\sqrt{2}}{2})\theta+2^{1/4}(\frac{11}{3}+\frac{3\sqrt{2}}{2})\theta^{1/2} \]
if $\theta > 4$. In both cases, it suffices to have $2r \geq \theta^2+9\theta^{3/2}$. This proves Theorem \ref{thm:main}. 
\end{proof}

\section{Concluding Remarks} \label{sec:conc}
There are further potential improvements to Theorem \ref{thm:main}. In particular, one would expect that the lower order terms can be reduced further. A number of variables are lost in Theorem \ref{thm:general} from the bounds coming from the use of Diophantine pairs. Therefore, one approach to reducing the lower order terms in Theorem \ref{thm:main} could be a more refined analysis of admissible Diophantine pairs. Lemma \ref{lem:conj2} shows that one expects to do better than what is available by an application of the Main Conjecture of Vinogradov's Mean Value Theorem. 

Let us examine \eqref{eq:diopairsystem}. One expects three kinds of solutions to this system. The first are diagonal solutions, the magnitude of which one expects to be $\ll Y^m$. The second are small solutions, which is to say $|x_i| \leq cY^{1-1/k}$ for some small constant $c$. There are $\ll Y^{2m(1-1/k)}$ of these. Finally, there are solutions characterized as a product of local densities, which number as $\ll Y^{2m-\frac{1}{2}(k-1)}$. If this heuristic is to be believed, one would make the following conjecture. 

\begin{conjecture} \label{conj:dioph}
Fix $k \in \mathbb{N}$ with $k \geq 2$. Fix also positive constants $C_0, \ldots, C_{k-1}$. Let $\mathcal{C}_{m,k}(X)$ be the set of integer solutions $(x_1,\ldots,x_{2m})$ to the system 
\[ |x_1^i +\cdots+x_{m}^i-x_{k}^i-\cdots-x_{2m}^i| \leq C_i Y^{i(1-1/k)} \quad \text{ for } \quad 1 \leq i \leq k-1, \]
with $|x_i| \leq C_0Y$. 
Then for any $\epsilon > 0$ one has the estimate 
\[ |\mathcal{C}_{m,k}(X)| \ll_{C_0,\ldots,C_{k-1}, \epsilon} Y^{\epsilon}( Y^{m}+Y^{2m(1-\frac{1}{k})}+ Y^{2m-\frac{1}{2}(k-1)}). \] 
In particular, one has $|\mathcal{C}_{m,k}(X)| \ll_{\epsilon} Y^{2m-\frac{1}{2}(k-1)+\epsilon}$ for $2m \geq k(k-1)/2$. 
\end{conjecture} 

Conjecture \ref{conj:dioph} essentially states that for this system, one only requires half of the number of variables to obtain the bound analogous to the Main Conjecture of Vinogradov's Mean Value Theorem, as indicated by Lemma \ref{lem:vinodiopair}. Perhaps it is possible to do even better. One might also consider the system 
\[ |x_1^i +\cdots+x_{m}^i-x_{k}^i-\cdots-x_{2m}^i| \leq C_i Y^{\delta i}, \quad 1 \leq i \leq k-1, \]
for other $0 \leq \delta \leq 1$. This system may arise in other Diophantine counting problems as well. 

Proving Conjecture \ref{conj:dioph} would reduce the bound in Theorem \ref{thm:main} to $2r \geq \theta^2 +C\theta^{3/2}$ for an even smaller constant $C < 9$. Similar bounds for $\widetilde{G}_0(\theta)$ and $ \widetilde{G}_0(\theta)$ would follow. While this result may be slightly underwhelming, it may be possible to extract further information from solutions to \eqref{eq:diopairsystem}. Inspecting the proof of Lemma \ref{lem:conj2}, we see that either the system displays strong diagonal behavior or a divisor estimate constrains the possible values of the variables. A similar phenomenon might hold in general. 
An analysis of the structure of the solutions to \eqref{eq:diopairsystem} rather than just the number may yield better bounds on the lower order terms in Theorem \ref{thm:main}, maybe even $2r \geq \theta^2 +C\theta$. For example, when $k=2$ and $m=1$, one knows that $x_1$ must be within $\ll X^{1/2}$ of $x_2$ and this stronger fact is used in the proof of Theorem 3.4 in \cite{Poulias1}. This allows for the removal of the $\epsilon$ as well. Using this as inspiration to develop a more efficient version of the argument in Lemma \ref{lem:meancasegen} might allow the recovery of the $2M_n$ variables lost in Theorem \ref{thm:general}.

\bibliographystyle{plain}
\bibliography{refs}

\noindent\textsc{Department of Mathematics, Purdue University, West Lafayette, IN, USA.}
\vspace{.03in}
\newline\noindent\textit{Email address}: bharga37@purdue.edu.

\end{document}